\documentclass[draft]{amsart}

\usepackage{amssymb}
\usepackage{url}
\usepackage{amscd}
\usepackage{enumerate}
\usepackage[all]{xy}

\def \N{{\mathbb N}}

\def \R{{\mathbb R}}

\def \1{{\mathbb 1}}

\theoremstyle{plain}
\newtheorem{theorem}{Theorem}
\newtheorem{proposition}{Proposition}
\newtheorem{definition}{Definition}
\newtheorem{lemma}{Lemma} 
\newtheorem{corollary}{Corollary}

\theoremstyle{remark}
\newtheorem{remark}{Remark}

\newtheorem{Exemp}{Example}

\begin{document}
\title[Index of symmetry of asymmetric normed space]{Index of symmetry and topological classification of asymmetric normed spaces.}
\author{M. Bachir$^*$, G. Flores$^\dagger$}
\date{23th February 2020}

\address{$^*$ Laboratoire SAMM 4543, Universit\'e Paris 1 Panth\'eon-Sorbonne\\ Centre P.M.F. 90 rue Tolbiac\\
75634 Paris cedex 13\\
France}
\email{Mohammed.Bachir@univ-paris1.fr}

\address{$^\dagger$ Departamento de Ingenier\'ia Matem\'atica, Universidad de Chile\\
Beauchef 851, Santiago\\
Chile}
\email{gflores@dim.uchile.cl}

\begin{abstract}
Let $X,Y$  be asymmetric normed spaces and $L_c(X,Y)$ the convex cone of all linear continuous operators from $X$ to $Y$. It is known that in general,  $L_c(X,Y)$ is not a vector space. The aim of this note is to give, using the Baire category theorem, a complete cracterization on $X$ and a finite dimensional $Y$ so that  $L_c(X,Y)$ is a vector space.  For this, we introduce an index of symmetry of the space $X$ denoted $c(X)\in [0,1]$ and we give the link between the index $c(X)$  and the fact that $L_c(X,Y)$ is in turn an asymmetric normed space for every asymmetric normed space $Y$. Our study leads to a topological classification of asymmetric normed spaces.
\end{abstract}
\maketitle

\noindent {\bf Keyword, phrase:} Asymmetric normed space, Index of symmetry, Baire category theorem, Hahn-Banach theorem, Topological classification. 
\section{Introduction}
An asymmetric normed space is a real vector space $X$ equipped with a positive, subadditive and positively homogeneous function $\|\cdot|_X$ satisfying: for all $x\in X$, $\|x|_X=\|-x|_X=0\Longleftrightarrow x=0$. (The difference between a classical seminorm and an asymmetric norm, is that the equality $\|-x|_X=\|x|_X$ does not always hold. An asymmetric norm is also called a quasi-norm in \cite{AFG1, AFG2} and \cite{RS}).  Every asymmetric normed space $(X,\|\cdot|_X)$  has an associated  normed space $X_s:=(X,\|\cdot\|_s)$ with the norm $\|x\|_s:=\max\{\|x|_X,\|-x|_X\}$, for all $x\in X$. A linear operator $T: (X,\|\cdot|_X)\to (Y,\|\cdot|_Y)$ between two asymmetric normed spaces is called continuous if it is continuous with respect to the topologies $\tau_X$ on $X$ and $\tau_Y$ on $Y$, induced respectively by the asymmetric norms. The set of all continuous linear operators from $X$ to $Y$ is denoted by $L_c(X, Y )$. The space of all linear continuous operators between the associated normed spaces $(X, \|\cdot\|_s)$ and $(Y, \|\cdot\|_s)$ is denoted by $L(X_s,Y_s)$.  In the case  $(Y,\|\cdot|_Y)=(\R,\|\cdot|_{\R})$, where $\|t|_{\R}:=\max\{0,t\}$  for all  $t\in \R$, we denote $X^{\flat}:=L_c(X,\R)$ called the dual of $(X,\|\cdot|)$, and $X^*:=L(X_s,\R_s)$, the classical topological dual of the normed space $(X, \|\cdot\|_s)$. For literature on asymmetric normed spaces, linear continuous operators  and their applications, we refer to \cite{GRS, GRS1}, \cite{Cob}, \cite{AFG1, AFG2} and \cite{RS}. Quasi-metric spaces and asymmetric norms have recently attracted a lot of interest in modern mathematics, they arise naturally when considering non-reversible Finsler manifolds \cite{CJ, DJV, JS}. For an introduction and study of asymmetric free spaces (or semi-Lipschitz free spaces), we refer to the recent paper \cite{DSV}.
\vskip5mm
The dual $X^{\flat}$ is formed by all linear continuous functionals from $ (X,\|\cdot|_X)$ to $(\R,\|\cdot|_{\R})$, or equivalently, by all linear upper semicontinuous functionals from $(X,\|\cdot|_X)$ to $(\R,|\cdot|)$. In contrast to the usual case, the dual $X^{\flat}$ and $L_c(X,Y)$ are not necessarily linear spaces, but merely convex cones contained respectively in $X^*$ and $L(X_s,Y_s)$. This being said, there is no characterization, in the literature, of the asymmetric normed spaces $X$ for which $L_c(X,Y)$  is also an asymmetric normed space. The purpose of this note is to give a complete answer to this problem when $Y$ is of finite dimensional (Corollary \ref{dimf}). It turns out that the answer to this question 
is linked to the{\it ``index of symmetry"} $c(X,\|\cdot|_X)$ of $X$ that we introduce in this note. This index is defined as follows (we will denote $c(X)$ instead of $c(X,\|\cdot|_X)$ when there is no ambiguity): 
$$c(X):=\inf_{\|x|_X=1} \|-x|_X \in [0,1].$$
This index measures the degree of symmetry of the asymmetric norm $\|\cdot|_X$. It is clear that an asymmetric norm is a norm if and only if $c(X)=1$. The asymmetry of $\|\cdot|_X$,  increases according to the decrease of the index $c(X)$. 

The aim of this note is to prove that the case $c(X)=0$ is exactly the situation where the convex cone $L_c(X, Y )$ (in particular the dual $X^{\flat}$) lacks vector structure, for every asymmetric normed space $Y$ which is not $T_1$ (i.e. for which there exists $e\in Y$ such that $\|e|_Y=1$ and $\|-e|_Y=0$, see Theorem ~\ref{Main}). In particular, we prove in Corollary \ref{dimf} the following caracterization, when $Y$ is of finite dimensional:  $L_c(X, Y ) \textnormal{  is not a vector space if and only if } c(X)=c(Y)=0.$

In consequence, from a topological point of view, the case where $c(X) = 0$ turns out to be the only interesting case in the theory of asymmetric normed spaces. Indeed, we prove in Corollary ~\ref{Car} that, $(a)$ $X^{\flat}$ is a vector space, if and only if,  $(b)$ $(X,\|\cdot|_X)$ is isomorphic to its associated normed space, if and only if, $(c)$ $L_c(X, Y )$ is a vector space isomorphic to $L(X_s,Y_s)$, for every asymmetric normed space $Y$, if and only if, $(d)$ $c(X)>0$. This means that the case where $c(X)>0$, refers to the classical framework of normed spaces. The most challenging implication is $(a) \Longrightarrow (b)$ and uses the Baire category theorem. These statements are consequences of our first main result, Theorem ~\ref{Main}. Our second main result (Theorem ~\ref{Tain}) shows that an asymmetric normed space $X$ is a $T_1$ space if and only if its dual $X^{\flat}$ is weak-star dense in $(X^*,w^*)$. 

Our study eventually leads to some topological classification of the asymmetric normed  spaces. There are three types of asymmetric normed spaces given in the following definition.

\begin{definition} \label{def1}  Let $X$ be an asymmetric normed space. We say that 

\begin{enumerate}[$(i)$]
    \item $X$ is of type I if $c(X)>0$ (necessarily a $T_1$ space).
    \item $X$ is of type II if $c(X)=0$ and $X$ is a $T_1$ space.
    \item $X$ is of type III if $X$ is not a $T_1$ space (necessarily $c(X)=0$).
\end{enumerate}

\end{definition}
\paragraph{\bf Comments} From a topological point of view, asymmetric normed spaces of type I present no new interest compared to the classical theory of normed spaces, since these spaces are isomorphic to their associated normed spaces and the same holds for their duals (see Corollary \ref{Car}). Moreover, the class of finite dimensional $T_1$ asymmetric normed spaces is contained in the class of type ~I (see Theorem \ref{ok3}). Spaces of type ~II and ~III are the interesting cases, since they differ from the framework of classical normed spaces. Spaces of type ~III include finite dimensional spaces (for example, $(\R,\|\cdot|_{\R})$), while spaces of type ~II only include infinite dimensional spaces (see Proposition \ref{ok2}). 

Briefly, the interest of asymmetric normed space theory, from a topological point of view, concerns only the following cases:

\begin{enumerate}[(i)]
    \item Infinite dimensional spaces which are $T_1$ with $c(X)=0$ (spaces of type ~II).
    \item Finite and infinite dimensional spaces $X$ which are not $T_1$ (spaces of type ~III, necessarily $c(X)=0$).
\end{enumerate}

Types II and III (corresponding to the case of $c(X)=0$) are exactly the situations where the dual $X^{\flat}$ is not a vector space. Moreover, an asymmetric normed space $X$ of type ~I will be always isomorphic to its associated normed space and $L_c(X,Y)\simeq L(X_s,Y_s)$ for every asymmetric normed space $Y$. Examples illustrating the three types of spaces will be given at the end of Section ~\ref{Cl}.

\vskip5mm
This paper is organized as follows. In Section ~\ref{P}, we recall definitions and notation from the literature. In Section ~\ref{ind}, we give basic properties of the index of symmetry. In Section ~\ref{Ma}, we state and prove our  main results (Theorem ~\ref{Main}, Corollary ~\ref{Car} and Theorem ~\ref{Tain}) and discuss some consequences. Finally, in Section ~\ref{Cl}, we give the proofs of Theorem ~\ref{ok3} and Proposition ~\ref{ok2} which motivate the aforementioned classification via the index of symmetry.
\section{Definitions and notion} \label{P}
In this section, we recall known properties of asymmetric  normed spaces that we are going to use  in the sequel.
\begin{definition}
Let $X$ be a real linear space. We say that $\|\cdot|: X \to \R^+$ is an asymmetric norm on $X$
if  the following properties hold.

\begin{enumerate}[$(i)$]
    \item For every $\lambda \geq 0$ and every $x \in X$, $\|\lambda x|= \lambda \|x|$.
    \item For every $x, y \in X$, $\|x+y|\leq \|x|+\|y|$.
    \item For every $x \in X$, if $\|x|=\|-x|=0$ then  $x=0$.
\end{enumerate}
 
\end{definition}
Let $(X,\|\cdot|_X)$ and $(Y,\|\cdot|_Y)$ be two asymmetric normed spaces. A linear operator $T: (X,\|\cdot|_X)\to (Y,\|\cdot|_Y)$ is called bounded if there exists $C\geq 0$ such that 
$$\|T(x)|_Y\leq C\|x|_X, \hspace{2mm} \forall x\in X.$$
In this case, we denote $\|T|_{L_c}:=\sup_{\|x|_X\leq 1}\|T(x)|_Y$. 
It is known (see \cite[Proposition 3.1]{Cob}) that a linear operator $T$ is bounded if and only if it is continuous, which in turn is equivalent to being continuous at $0$. Also, we know from \cite[Proposition 3.6]{Cob} that the constant $\|T|_{L_c}$ can be calculated also by the formula
$$\|T|_{L_c}=\sup_{\|x|_X= 1}\|T(x)|_Y.$$
 We can see that $L_c(X, Y )$ is a convex cone included in $L(X_s,Y_s)$ (see \cite[Proposition 3.3]{Cob}) but is not a vector space in general. Note that for each $T\in L_c(X,Y)$ we have that 
$$\|T\|_{L_s}\leq \|T|_{L_c},$$
where $\|\cdot\|_{L_s}$, denotes the usual norm of $L(X_s,Y_s)$. The function $\|\cdot|_{L_c}$ defines an asymmetric norm on $L_c(X, Y )\cap(-L_c(X, Y ))$. Recall that in the case where $Y=(\R,\|\cdot|_{\R})$, where $\|t|_{\R}=\max\{0,t\}$ for all $t\in \R$, we denote $X^{\flat}:=L_c(X,\R)$, called the dual of the asymmetric normed space $X$.  The topological dual of the associated normed space $X_s:=(X,\|\cdot\|_s)$ of $X$  is denoted $X^*$ and is equipped with the usual dual norm denoted $\|p\|_{*}=\sup_{\|x\|_s\leq 1} \langle p, x\rangle$, for all $p\in X^*$.  Note, from \cite[Theorem 2.2.2]{Cob}, that the convex cone $X^{\flat}$ is not trivial, that is, $X^{\flat}\neq \lbrace 0 \rbrace$ whenever $X\neq\lbrace 0 \rbrace$. We always have that
$$X^{\flat} \subset X^* \textnormal{ and }  \|p\|_{*}\leq \|p|_{\flat}, \textnormal{ for all } p\in X^{\flat}.$$ 
We say that $(X,\|\cdot|_X)$ and $(Y,\|\cdot|_Y)$ are isomorphic and we use the notation $(X,\|\cdot|_X) \simeq (Y,\|\cdot|_Y)$, if there exists a bijective linear operator $T: (X,\|\cdot|_X)\to (Y,\|\cdot|_Y)$ such that $T$ and $T^{-1}$ are bounded. 
\vskip5mm
\section{Index of an asymmetric normed space}  \label{ind}
We give the following properties of the index $c(X)$ of an asymmetric normed space $X$. Denote $$S_X:=\lbrace x\in X:  \|x|=1\rbrace,$$ and  $$\hat{S}_X:=\lbrace x\in S_X: \|-x|\neq 0\rbrace.$$
It is not difficult to see that 
$$\hat{S}_X=S_X \textnormal{ if and only if } \forall x\in X, \|x|=0\Longleftrightarrow x=0.$$
Recall that a topological space $X$ is called $T_1$ if for every pairs of points $(x,y)$, there exists an open set of $X$ containing $x$ and not containing $y$. It is known that a topological space $X$ is $T_1$ if and only if for every $x\in X$, the singleton $\lbrace x \rbrace$ is closed. We have the following elementary propositions.
\begin{proposition} \label{ok1}  Let $(X,\|\cdot|_X)$  be an asymmetric normed space. Then, $X$ is a $T_1$ space if and only if $\hat{S}_X = S_X$. 
\end{proposition}
\begin{proof} Suppose that $X$ is not $T_1$. Then, there exists $a\in X$ such that $X\setminus \lbrace a \rbrace$ is not open. Thus, there exists $b\in X\setminus \lbrace a \rbrace$ such that for every $\varepsilon >0$, $a\in B_{\|\cdot|_X}(b,\varepsilon)$. In other words, we have that $\|a-b|_X=0$. Thus, we have that $a-b\neq 0$ and $\|a-b|_X=0$. It follows that $\|b-a|_X\neq 0$. Let us set $e:=\frac{b-a}{\|b-a|_X}$. Then, we have that $e\in S_X$ and $\|-e|_X=0$. Hence, $\hat{S}_X \neq S_X$. Conversely, suppose that $\hat{S}_X \neq S_X$ and let $e\in S_X$ be such that $\|-e|_X=0$. This implies that the singleton $\lbrace 0 \rbrace$ is not closed in $X$. Hence, $X$ is not $T_1$.
\end{proof}
\begin{proposition} \label{index} Let $(X,\|\cdot|)$  be an asymmetric normed space. Suppose  that $c(X)>0$. Then, $X$ is a $T_1$ space (equivalently, $\hat{S}_X=S_X$). Moreover, we have 
$$\frac{1}{c(X)}=\sup_{x\in S_X }\|-x|=\sup_{\|x|=1, \|p|_{\flat}=1; \langle -p, x\rangle >0}\langle -p, x\rangle.$$
Therefore, we have that
$$(\sup_{x\in S_X}\|-x|)(\inf_{x\in S_X}\|-x|)=1,$$
and so $c(X)\in [0,1]$.
\end{proposition}
\begin{proof} Let $x\in X$ such that $\|x|\neq 0$, then $\frac{x}{\|x|}\in S_X$ and $\|\frac{-x}{\|x|}|\geq c(X)>0$. It follows that $\|-x|\neq 0$. Equivalently, $\|-x|=0\Longrightarrow \|x|=0$ and so $x=0$. It follows that $\hat{S}_X=S_X$ and so $X$ is a $T_1$ space. On the other hand, 
$$\forall x\in X, (x\in S_X  \Longleftrightarrow \exists z\in S_X , x=-\frac{z}{\|-z|}).$$
Indeed, it sufices to take $z=-\frac{x}{\|-x|}$. From this,  we have that
$$\sup_{x\in S_X}\|-x|=\sup_{z\in S_X} (\frac{1}{\|-z|})=\frac{1}{\inf_{z\in S_X}\|-z|}=\frac{1}{c(X)}.$$  
Now, by the Hahn-Banach theorem (see \cite[Corollary 2.2.4]{Cob}), we have
$$\sup_{\|x|=1}\|-x|=\sup_{\|x|=1}\sup_{\|p|_{\flat}=1} \langle p, -x\rangle=\sup_{\|x|=1, \|p|_{\flat}=1; \langle -p, x\rangle> 0}\langle -p, x\rangle.$$
\end{proof}
\begin{proposition}  \label{tex}  Let $(X,\|\cdot|_X)$ and $(Y,\|\cdot|_Y)$ be asymmetric normed spaces. Suppose  that $c(X)>0$. Then, we have the following formulas:

\begin{eqnarray} \label{eq1}
c(X)\|x|\leq \|-x|\leq \frac{1}{c(X)}\|x|, \hspace{2mm} \forall x\in X,
\end{eqnarray}
\begin{eqnarray} \label{eq2}
c(X)\|x\|_s\leq \|x|\leq \|x\|_s, \hspace{2mm} \forall x\in X, 
\end{eqnarray}
\begin{eqnarray} \label{eq21}
c(X)\|T|_{L_c}\leq \|-T|_{L_c}\leq \frac{1}{c(X)} \|T|_{L_c}, \hspace{2mm} \forall T\in L_c(X,Y),
\end{eqnarray}
and,
\begin{eqnarray} \label{eq3}
\|T\|_{L_s}\leq \|T|_{L_c}\leq \frac{1}{c(X)}\|T\|_{L_s}, \hspace{2mm} \forall T\in L_c(X,Y).
\end{eqnarray}
Consequently, we have that $(X,\|\cdot|_X)\simeq (X,\|\cdot\|_s)$, $( L_c(X,Y),\|\cdot|_{L_c})$ is asymmetric normed space and $( L_c(X,Y),\|\cdot|_{L_c})\simeq ( L(X_s,Y_s),\|\cdot\|_{L_s})$. Moreover, $c( L_c(X,Y))\geq c(X)$.

\end{proposition}
\begin{proof} Formula (\ref{eq1}) follows easily from Proposition \ref{index}, the rest of the assertions are simple consequences of this formula.
\end{proof}
The above proposition says that if $c(X)>0$, then  $(X,\|\cdot|_X)$ and $(X,\|\cdot\|_s)$ (and also $( L_c(X,Y),\|\cdot|_{L_c})$ and $( L(X_s,Y_s),\|\cdot\|_{L_s})$) are topologically the same. In this case, the topology of the asymmetric spaces coincides with the topology of  normed spaces.
\begin{remark} We proved in Proposition \ref{index} that if $c(X)>0$, then $\hat{S}_X =S_X$. The converse of this fact is not true in general, see for instance Example \ref{ex}.
\end{remark}


\section{The main results}\label{Ma}
This section is devoted to the main results Theorem \ref{Main}, Theorem \ref{Tain}, Corollary \ref{dimf} and their consequences.

Let $(X,\|\cdot|_X)$  and $(Y,\|\cdot|_Y)$ be asymmetric normed spaces. For each $r>0$, denote $$B_{L_c}(0,r):=\lbrace T\in L_c(X,Y): \|T|_{L_c}< r\rbrace$$ and $$\overline{B}_{L_c}(0,r):=\lbrace T\in L_c(X,Y): \|T|_{L_c}\leq r\rbrace.$$ 
Recall that $(Y,\|\cdot|_Y)$ is called a biBanach asymmetric normed space, if  its associated normed space $(Y,\|\cdot\|_s)$ is a Banach space.  In this case, the space $(L(X_s,Y_s),\|\cdot\|_{L_s})$  is also a Banach space. 
We need the following lemma.
\begin{lemma} \label{lema}  Let $(X,\|\cdot|_X)$ be an asymmetric normed space and $(Y,\|\cdot|_Y)$ be an asymmetric normed biBanach space.  Let $r>0$, then $\overline{B}_{L_c}(0,r)$ is a closed subset of the Banach space $(L(X_s,Y_s),\|\cdot\|_{L_s})$ and the open ball $B_{L_c}(0,r)$ is dense in $\overline{B}_{L_c}(0,r)$ for the norm $\|\cdot\|_{L_s}$.
\end{lemma}
\begin{proof} Let $(T_n)$ be a sequence in $\overline{B}_{L_c}(0,r)$ that converges to $T\in L(X_s,Y_s)$ for the norm $\|\cdot\|_{L_s}$ and let us prove that $T\in \overline{B}_{L_c}(0,r)$. Indeed, for all $n\in \N$ and all $x\in X$, we have 
\begin{eqnarray*}
\|T(x)|_Y &\leq& \|(T-T_n)(x)|_Y +\|T_n(x)|_Y\\
                              &\leq& \|T-T_n\|_{L_s}\|x\|_s+\|T_n|_{L_c}\|x|_X\\
                              &\leq& \|T-T_n\|_{L_s}\|x\|_s+r\|x|_X.
\end{eqnarray*}
Sending $n$ to $+\infty$, we get that $\|T(x)|_Y \leq r\|x|_X$, for all $x\in X$. It follows that $T\in \overline{B}_{L_c}(0,r)$ which implies that $\overline{B}_{L_c}(0,r)$ is a  closed subset of the space $(L(X_s,Y_s),\|\cdot\|_{L_s})$. 
To see that $B_{L_c}(0,r)$ is dense in $\overline{B}_{L_c}(0,r)$ for the norm $\|\cdot\|_{L_s}$, let $T\in \overline{B}_{L_c}(0,r)$ and consider the sequence $T_n=(1-\frac{1}{n})T$ so that $\|T_n|_{L_c}\leq (1-\frac{1}{n})r< r$. Then, $T_n \in B_{L_c}(0,r)$ for all $n \in \N$ and $\|T-T_n\|_{L_s}=\frac{1}{n}\|T\|_{L_s}\to 0$.
\end{proof}
\begin{remark} Since $\overline{B}_{L_c}(0,r)$ is a closed subset of $(L(X_s,Y_s),\|\cdot\|_{L_s})$, then $(\overline{B}_{L_c}(0,r), \|\cdot\|_{L_s})$ is a complete metric space and so the Baire cathegory theorem applies. However, we dont know if  the whole space $(L_c(X,Y), \|\cdot\|_{L_s})$ is a Baire space. In general it is not closed in $(L(X_s,Y_s),\|\cdot\|_{L_s})$ (see Corollary ~\ref{clos} in the case where $(Y,\|\cdot|_Y)=(\R,\| \cdot|_{\R}))$. 
\end{remark}
\subsection{The first main result and consequences.}
Our first main result is the following theorem, which gives a necessary and sufficient condition so that $L_c(X,Y)$ is not a vector space under the assumption that $Y$ is not a $T_1$ space. This frame corresponds to the case where $c(X) = 0$, and it is this case that makes the theory of asymmetric normed spaces interesting. 
\begin{remark}\label{remv1} $(i)$ Note that our work in this paper is of interest even in the case where $(Y,\|\cdot|_Y)=(\R,\|\cdot|_{\R})$ and the goal is primarily to characterize asymmetric normed spaces $X$ for which $X^{\flat}$ is a vector space. This seems to be non-trivial without the help of the Baire category theorem.

$(ii)$  Let $(X,\|\cdot|_X)$ and $(Y,\|\cdot|_Y)$ be asymmetric normed spaces. In general, the convex cone $ X^{\flat}$ not embedds in $L_c(X,Y)$.  However, if we assume that there exists $e\in Y$  such that $\|e|_Y=1$ and $\|-e|_Y=0$ (that is $Y$ is not a $T_1$ space, see Proposition \ref{ok1}) then, for every $p\in X^{\flat}$ we have that $pe\in L_c(X,Y)$ and $\|pe|_{L_c}=\|p|_{\flat}$, where $pe: x\mapsto \langle p, x \rangle e$. This is due to the fact that, for each $x\in X$, $\|p(x)e|_Y=p(x)$ if $p(x)\geq 0$ and $\|p(x)e|_Y=0$ if $p(x)\leq 0$. In other words, the convex cone $(X^{\flat},\|\cdot|_{\flat})$ embedds isometricaly in $(L_c(X,Y),\|\cdot|_{L_c})$, whenever $Y$ is not a $T_1$ space. This remark will be used in the proof of Theorem \ref{Main}.

\end{remark}
\begin{theorem} \label{Main} Let $(X,\|\cdot|_X)$  be asymmetric normed space. Then, the following assertions are equivalent.

\begin{enumerate}[$(i)$]
    \item $c(X)=0$.
    \item For every  biBanach asymmetric normed space $(Y,\|\cdot|_Y)$ for which there exists $e\in Y$ such that $\|e|_Y=1$ and $\|-e|_Y=0$ (that is, $Y$ is not a $T_1$ space) and every $H\in L_c(X,Y)$ and every $r>0$, the set $$\mathcal{G}(H):=\lbrace T\in \overline{B}_{L_c}(0,r): -(H+T)\not \in L_c(X,Y)\rbrace,$$
is a $G_\delta$ dense subset of $(\overline{B}_{L_c}(0,r), \|\cdot\|_{L_s})$. In particular, $L_c(X,Y)$ is not a vector space for every  biBanach asymmetric normed space $(Y,\|\cdot|_Y)$ which is not a $T_1$ space.
    \item There exists an asymmetric normed space $(Y,\|\cdot|_Y)$ such that the convex cone  $L_c(X,Y)$ is not a vector space.
\end{enumerate}

\end{theorem}
\begin{proof}   First, let us note that we can assume without loss of generality that $r=1$. 

$(i) \Longrightarrow (ii)$ For each $k\in \N$, let us set
\begin{eqnarray*}
O_k:=\lbrace T\in \overline{B}_{L_c}(0,1)\,|\, (\exists x_k\in X)\,  \|-(H+T)(x_k)|_Y> k\|x_k|_X \rbrace.
\end{eqnarray*}
 Clearly, we have that $\cap_{k\in \N} O_k=\mathcal{G}(H)$. By the Baire theorem, $\mathcal{G}(H)$ will be a $G_\delta$ dense subset of $(\overline{B}_{L_c}(0,1), \|\cdot\|_{L_s})$ whenever, for each $k\in \N$, the set $O_k$ is open and dense in the complete metric space $(\overline{B}_{L_c}(0,1), \|\cdot\|_{L_s})$ (see Lemma \ref{lema}).

Let us prove that $O_k$ is open in $(\overline{B}_{L_c}(0,1), \|\cdot\|_{L_s})$, for each $k\in \N$. Let $T\in O_k$ and $0< \varepsilon< \frac{\|-(H+T)(x_k)|_Y - k\|x_k|_X}{\|x_k\|_s}$. Let $S\in \overline{B}_{L_c}(0,1)$ such that $\|S-T\|_{L_s}<\varepsilon$.  We have that
\begin{eqnarray*}
\|-(H+S)(x_k)|_Y&\geq&  \|-(H+T)(x_k)|_Y -\|(S-T)(x_k)|_Y \\
                                       &\geq& \|-(H+T)(x_k)|_Y -\|(S-T)(x_k)\|_s  \\
                                        &\geq& \|-(H+T)(x_k)|_Y -\|S-T\|_{L_s}\|x_k\|_s\\
                                       &>& \|-(H+T)(x_k)|_Y -\varepsilon\|x_k\|_s\\
                                       &>&  k\|x_k|_X.
\end{eqnarray*}
Thus, $S\in O_k$ for every $S\in \overline{B}_{L_c}(0,1)$ such that $\|S-T\|_{L_s}<\varepsilon$. Hence, $O_k$ is open in $(\overline{B}_{L_c}(0,1), \|\cdot\|_{L_s})$.

Now, let us prove that $O_k$ is dense in $(\overline{B}_{L_c}(0,1), \|\cdot\|_{L_s})$, for each $k\in \N$. Since $(B_{L_c}(0,1), \|\cdot\|_{L_s})$ is dense in $(\overline{B}_{L_c}(0,1), \|\cdot\|_{L_s})$ (by Lemma \ref{lema}), it suffices to prove that $O_k$ is dense in $(B_{L_c}(0,1), \|\cdot\|_{L_s})$.  Let  $T\in B_{L_c}(0,1)$ and $0<\varepsilon< 1-\|T|_{L_c}$. Since $c(X)=0$, there exists a sequence $(a_n)\subset X$  such that $\|a_n|_X=1 $ for all $n\in \N$ and  $\|-a_n|_X\to  0$. Let us set $I:=\lbrace n\in \N: \|-a_n|_X=0\rbrace$. Recall that since $\|e|_Y=1$ and $\|-e|_Y=0$, we  have that $\|p e|_{L_c}=\|p|_{\flat}$ for every $p\in X^{\flat}$, where $pe: x\mapsto \langle p,x \rangle e$ (see Remark \ref{remv1} $(ii)$). We have two cases:
\vskip5mm
{\it Case 1. $I=\emptyset$.} In this case, for all $n\in \N$,  let $z_n:=\frac{-a_n}{\|-a_n|_X}$. We see that $\|z_n|_X=1$, $-z_n=\frac{a_n}{\|-a_n|_X}$ and $\|-z_n|_X=\frac{1}{\|-a_n|_X}$.
Using the Hahn-Banach theorem \cite[Theorem 2.2.2]{Cob}, for each $n\in \N$, there exists $p_n\in X^{\flat}$ such that $\|p_n|_{\flat}=1$ and $\langle p_n, -z_n\rangle =\|-z_n|_X>0$. Now,  consider the operators 
\begin{eqnarray*}
T+\varepsilon p_n e: (X,\|\cdot|_X) &\to& (Y,\|\cdot|_Y)\\
                                   x &\mapsto& T(x)+\varepsilon \langle p_n, x\rangle e.
\end{eqnarray*}
Thus, we have, $\|T+\varepsilon p_{n}e|_{L_c}\leq\|T|_{L_c}+ \varepsilon\|p_n e|_{L_c}=\|T|_{L_c}+ \varepsilon\|p_n|_{\flat}=  \|T|_{L_c}+ \varepsilon<1$, so that $T+\varepsilon p_{n} e\in B_{L_c}(0,1)\subset \overline{B}_{L_c}(0,1)$. On the other hand,
\begin{eqnarray*}
\|-(H+T+\varepsilon p_n e)(z_n)|_Y&=&\|(H+T+\varepsilon p_n e)(-z_n)|_Y\\
                                                               &\geq& \|\varepsilon p_n(-z_n) e|_Y - \|(H+T)(z_n)|_Y\\
                                                               &=&  \varepsilon\|-z_n|_X- \|(H+T)(z_n)|_Y\\
                                                               &=& \frac{\varepsilon}{\|-a_n|_X} - \|(H+T)(z_n)|_Y\\
                                                               &\geq& \frac{\varepsilon}{\|-a_n|_X}- \|H+T|_{L_c}.
\end{eqnarray*}
Since  $\|-a_n|_X\to 0$, when $n\to +\infty$, there exists a subsequence $(a_{n_k})$ such that $\frac{\varepsilon}{\|-a_{n_k}|_X}- \|H+T|_{L_c} > k$ for each $k\in \N$. Hence,  for each $k\in \N$, we have that $\|z_{n_k}|_X=1$  and
\begin{eqnarray} \label{eq*}
\|-(H+T+\varepsilon p_{n_k} e)(z_{n_k})|_Y&>&  k=k\|z_{n_k}|_X.
\end{eqnarray}
 From  formula (\ref{eq*}), we have that $T+\varepsilon p_{n_k}e\in O_k$.  Since, $$\|(T+\varepsilon p_{n_k} e)-T\|_{L_s}=\|\varepsilon p_{n_k} e\|_{L_s}=\varepsilon\|p_{n_k}\|_* \leq \varepsilon\| p_{n_k}|_{\flat}=\varepsilon,$$ it follows that $O_k$ is dense in $(\overline{B}_{L_c}(0,1), \|\cdot\|_{L_s})$. 
\vskip5mm
{\it Case 2. $I\neq \emptyset$.} In this case,  there exists $n_0\in I$ such that $\|a_{n_0}|_X=1$ and $\|-a_{n_0}|_X=0$. Using the Hahn-Banach theorem \cite[Theorem 2.2.2]{Cob},  let $p\in X^{\flat} \setminus \lbrace 0 \rbrace$ such that $\|p|_{\flat}=1$ and $\langle p,  a_{n_0}\rangle =\|a_{n_0}|_X=1$. Thus, we have that
\begin{eqnarray*}
\|-(H+T+\varepsilon p e)(-a_{n_0})|_Y&=&\|(-H-T-\varepsilon p e)(-a_{n_0})|_Y \\
                                                               &\geq& \|\varepsilon \langle- p, -a_{n_0} \rangle e|_Y - \|-(H+T)(a_{n_0})|_Y\\
                                                               &=& \varepsilon\|\langle p, a_{n_0} \rangle e|_Y- \|(H+T)(-a_{n_0})|_Y\\
                                                               &\geq& \varepsilon - \|H+T|_{L_c}\|-a_{n_0}|_X\\
                                                               &=& \varepsilon \\
                                                               &>& 0=k\|-a_{n_0}|_X.
\end{eqnarray*}
On the other hand, we have that $\|T+\varepsilon pe|_{L_c}\leq\|T|_{L_c}+ \varepsilon\|p e|_{L_c}=\|T|_{L_c}+ \varepsilon\|p|_{\flat}=  \|T|_{L_c}+ \varepsilon<1$, so that $T+\varepsilon p e\in B_{L_c}(0,1)\subset \overline{B}_{L_c}(0,1)$. Thus, $T+\varepsilon p e \in O_k$ and $\|(T+\varepsilon p e) -T\|_{L_s}=\|\varepsilon p e\|_{L_s}=\varepsilon \|p\|_{*}\leq \varepsilon \|p|_{\flat}= \varepsilon$. Hence, $O_k$ is dense in $(\overline{B}_{L_c}(0,1), \|\cdot\|_{L_s})$. 
\vskip5mm
Hence, in both cases, we have that $\cap_{k\in \N} O_k=\mathcal{G}(H)$ is a $G_\delta$ dense subset of $(\overline{B}_{L_c}(0,1),\|\cdot\|_{L_s}) $.
\vskip5mm
$(ii) \Longrightarrow (iii)$ is trivial.

$(iii) \Longrightarrow (i)$ Suppose that there exists an asymmetric normed space $Y$ such that the convex cone  $L_c(X,Y)$ is not a vector space. Then, there exists $T\in L_c(X,Y) \setminus \lbrace 0 \rbrace$, such that $-T\not \in L_c(X,Y)$. Thus, for each $n\in \N$, there exists $x_n\in X$ such that $\|-T(x_n)|_Y> n\|x_n|_X$. It follows that, for all $n\in \N$ 
$$n\|x_n|_X<\|-T(x_n)|_Y=\|T(-x_n)|_Y \leq \|T|_{L_c}\|-x_n|_X.$$
 Let us set $z_n=\frac{-x_n}{\|-x_n|_X}$ (since $\|-x_n|_X\neq 0$) for all $n\in \N$. Then,  for all $n\in \N\setminus \lbrace 0\rbrace$, we have $\|z_n|_X=1$, $\|-z_n|=\frac{\|x_n|}{\|-x_n|}< \frac{\|T|_{L_c}}{n}\to 0$. Hence, $c(X)=0$.
\end{proof}
\begin{remark} It does not seem obvious to show part $(i) \Longrightarrow (iii)$ directly without using the Baire theorem, in other words, without passing through part $(i) \Longrightarrow (ii)$. In fact the real difficulty is even in one dimensional, that is, to proof  the following implication
$$c(X)=0 \Longrightarrow \exists p \in X^{\flat} \textnormal{ s.t } -p \not \in X^{\flat}.$$
\end{remark}
A consequence of  Theorem \ref{Main}, under the condition that $Y$ is not a $T_1$ space, is:  $ L_c(X,Y)$ is not a vector space if and only if $c(X)=0$. Note that  the condition that $Y$ is not a $T_1$ space  implies trivially that $c(Y)=0$. The converse is not true in general (for example if $Y$ is the space $( l^{\infty}(\N^*),\|\cdot|_{\infty})$ in Example \ref{ex} below). However, these two conditions are equivalent when $Y$ is of finite dimensional.
\begin{lemma} \label{T1} Let $(Y,\|\cdot|_Y)$ be an asymmetric normed space of finite dimensional. Then, $c(Y)=0$ if and only if $Y$ is not a $T_1$ space.
\end{lemma}
\begin{proof}
Recall from Proposition \ref{ok1} that  $Y$ is not a $T_1$ space if and only if there exists $e\in Y$ such that $\|e|_Y=1$ and $\|-e|_Y=0$. The "if" part is clear. Let us prove the "only if" part.  Suppose that $c(Y)=0$, thus there exists a sequence $(y_n)\subset Y$ such that  $\|y_n|_Y=1$ for all $n\in \N$ and $\|-y_n|_Y\to 0$. We can assume without loss of generality that $\|-y_n|_Y <1$ for all $n\in \N$ so that we have $\|y_n\|_s=\|y_n|_Y=1$ for all $n\in \N$. Since $(Y,\|\cdot\|_s)$ is of finite dimensional, there exists a subsequence $(y_{n_k})$ converging for, the associated  norm $\|\cdot\|_s$, to some $e\in Y$ such that $\|e\|_s=1$. We show that $\|-e|_Y=0$. Indeed, 
$$\|-e|_Y\leq \|y_{n_k}-e|_Y+\|-y_{n_k}|_Y\leq \|y_{n_k}-e\|_s+\|-y_{n_k}|_Y \to 0.$$
Thus, $\|-e|_Y=0$ and $\|e|_Y=\|e\|_s=1$. Hence, $Y$ is not a $T_1$ space.
\end{proof}
 Thus, we obtain in the following result a complete characterization so that $ L_c(X,Y)$ is a vector space, in the case where $Y$ is of finite dimensional.  
\begin{corollary} \label{dimf}  Let $(X,\|\cdot|_X)$  be an asymmetric normed space and $(Y,\|\cdot|_Y)$ be an asymmetric normed space of finite dimensional. Then, $ L_c(X,Y)$ is not a vector space if and only if $c(X)=c(Y)=0$. The "only if part" is always true ,even if $Y$ is of infinite dimensional. 
\end{corollary}
\begin{proof}  To see the "only if" part,  we follow the proof of part $(iii) \Longrightarrow (i)$ of Theorem \ref{Main}. Indeed,  suppose that $L_c(X,Y)$ is not a vector space. Then, there exists $T\in L_c(X,Y) \setminus \lbrace 0 \rbrace$, such that $-T\not \in L_c(X,Y)$. Thus, for each $n\in \N$, there exists $x_n\in X$ such that $\|-T(x_n)|_Y> n\|x_n|_X$. It follows that, for all $n\in \N$ 
$$n\|x_n|_X<\|-T(x_n)|_Y=\|T(-x_n)|_Y \leq \|T|_{L_c}\|-x_n|_X.$$
 Let us set $z_n=\frac{-x_n}{\|-x_n|_X}$ (since $\|-x_n|_X\neq 0$) for all $n\in \N$. Then,  for all $n\in \N\setminus \lbrace 0\rbrace$, we have $\|z_n|_X=1$, $\|-z_n|=\frac{\|x_n|}{\|-x_n|}< \frac{\|T|_{L_c}}{n}\to 0$. Hence, $c(X)=0$. It remains to shows that $c(Y)=0$. Indeed, since $T\in L_c(X,Y) \setminus \lbrace 0 \rbrace$, then $\|T(x_n)|_Y\leq \|T|_{L_c}\|x_n|_X$. Thus, using the above inequality, we get that 
$$ \|T(x_n)|_Y < \frac{\|T|_{L_c}}{n}\|-T(x_n)|_Y.$$
This implies in particular that $\|-T(x_n)|_Y\neq 0$ for all $n\in \N \setminus \lbrace 0\rbrace$. Let us set $y_n:=\frac{-T(x_n)}{\|-T(x_n)|_Y}\in Y$. Then we have that $\|y_n|_Y=1$ for all $n\in \N \setminus \lbrace 0\rbrace$ and $\|-y_n|_Y < \frac{\|T|_{L_c}}{n}\to 0$. Hence $c(Y)=0$.

The "if" part will follow from Theorem \ref{Main} provided that the condition $c(Y)=0$ implies that there exists $e\in Y$ such that $\|e|_Y=1$ and $\|-e|_Y=0$. This is true since $Y$ is of finite dimensional (necessarilly a biBanach space), by Lemma \ref{T1}.
\end{proof}
\begin{remark} \label{open}
Combining Theorem \ref{Main} and  the "only if part" of Corollary \ref{dimf} we have that for every asymmetric normed spaces $(X,\|\cdot|_X)$ and $(Y,\|\cdot|_Y)$ (with $Y$ biBanach): $$ c(X)=0 \textnormal{ and } Y \textnormal{ is not } T_1 \Longrightarrow L_c(X,Y) \neq - L_c(X,Y) \Longrightarrow c(X)=c(Y)=0.$$
By Corollary \ref{dimf} and Lemma \ref{T1},  the reverse implications are also true  when the space $Y$ is of finite dimensional. We show in the following two examples that the reverse implications are in general false when $Y$ is of infinite dimensional. The authors are grateful to the referee for having communicated to them the following example \ref{referee} (i).
\end{remark}

\begin{Exemp} \label{referee} $(i)$ Let $X = \R^2$ and let the asymmetric norm $$\|(x_1,x_2)|_X=\max \lbrace |x_1|,x_2^+\rbrace,$$ where $x_2^+=\max \lbrace y_1, 0 \rbrace$. Clearly $c(X)=0$. Let $(Y,\|\cdot|_Y)$ be any $T_1$ space with $c(Y)=0$ (such space exists, see for example the  biBanach space $(Y,\|\cdot|_Y)=( l^{\infty}(\N^*),\|\cdot|_{\infty})$ in Example \ref{ex} below). Then, we have that $c(X)=c(Y)=0$ but $L_c(X,Y)$ is a vector space. Indeed, let $T\in L_c(X,Y)$. There exists $e_1,e_2\in Y$ such that $T(x_1,x_2)=x_1e_1+x_2e_2$ for all $(x_1,x_2)\in X$. Thus,
$$\|T(0,-1)|_Y=\|-e_2|_Y \leq \|T|_{L_c}\|(0,-1)|_X=0.$$
Since $Y$ is a $T_1$ space, we have that $e_2=0$. Thus,  $T(x_1,x_2)=x_1e_1$ for all $(x_1,x_2)\in X$. So, we have that for all $(x_1,x_2)\in X$
$$\|-T(x_1,x_2)|_Y=\|-x_1e_1|_Y \leq |x_1|\|e_1\|_s\leq \|e_1\|_s\|(x_1,x_2)|_X.$$
This shows that $-T\in L_c(X,Y)$. Hence, $L_c(X,Y)$ is a vector space.

$(ii)$ Let $(X,\|\cdot|_X)$  be an asymmetric normed space such that $X$ is a $T_1$ space with $c(X)=0$ (for example $(X,\|\cdot|_X)=( l^{\infty}(\N^*),\|\cdot|_{\infty})$ in Example \ref{ex} below). Then, the identity map $I\in  L_c(X,X)$ but $-I\not \in  L_c(X,X)$. Indeed, since $c(X)=0$ and $X$ is a $T_1$ space, then there exists a sequence $(x_n)\subset X$ such that $\|x_n|_X=1$ for all $n\in \N$ and $0<\|-x_n|_X\to 0$. Let $e_n:=\frac{-x_n}{\|-x_n|_Y}$, so $\frac{\|-I(e_n)|_Y}{\|e_n|_Y}=\frac{1}{\|-x_n|_Y}\to +\infty$. Thus, $L_c(X,X)$ is not a vector space, $c(X)=0$ but $X$ is a $T_1$ space.
\end{Exemp}

\vskip1mm
In the following result we give a density result for the asymmetric norm $\|\cdot|_{L_c}$.
\begin{corollary} \label{dense} Let $(X,\|\cdot|_X)$  be asymmetric normed space with $c(X)=0$ and  $(Y,\|\cdot|_Y)$ be a biBanach asymmetric normed space for which there exists $e\in Y$ such that $\|e|_Y=1$ and $\|-e|_Y=0$. Then, the set of elements $H\in L_c(X,Y)$ such that $-H\not \in L_c(X,Y)$ is dense in $L_c(X,Y)$ for the asymmetric norm $\|\cdot|_{L_c}$.
\end{corollary}
\begin{proof} Using Theorem \ref{Main} with $r=\varepsilon$ for every $\varepsilon>0$,  we get that for every $H\in L_c(X,Y)$, there exists $T\in \overline{B}_{L_c}(0,\varepsilon)$, such that $-(H+T)\not \in L_c(X,Y)$, however $H+T \in L_c(X,Y)$. 
\end{proof}
\begin{Exemp} \label{ex}   Let $X= l^{\infty}(\N^*)$ equipped with the asymmetric norm $\|\cdot|_{\infty}$ defined by 
$$\|x|_{\infty}=\sup_{n\in \N^*}\|x_n|_{\frac 1 n}\leq \|x\|_{\infty},$$
where for each $t\in \R$ and each $n\in \N^*$, $\|t|_{\frac 1 n}=t$  if $t\geq 0$ and $\|t|_{\frac 1 n}=-\frac{t}{n}$ if $t\leq 0$. Then, clearly $\hat{S}_X= S_X$ since $\|x|_{\infty}=0 \Longleftrightarrow \|-x|_{\infty}=0 \Longleftrightarrow x=0$. Thus, $(X,\|\cdot|_{\infty})$ is a $T_1$ asymmetric normed space. On the other hand, for each $n\in \N^*$, we have $\|e_n|_{\infty}=1$ and $\|-e_n|_{\infty}=\frac 1 n$, where $(e_n)$ is the canonical basis of $c_0(\N^*)$. It follows that $c(l^{\infty}(\N^*))=0$ and so, the set 
$$\lbrace p\in (l^{\infty}(\N^*))^{\flat} \textnormal{ such that } -p \not \in  (l^{\infty}(\N^*))^{\flat} \rbrace$$ is dense in  $(l^{\infty}(\N^*))^{\flat}$ for the asymmetric norm $\|\cdot|_{\flat}$ (see Corollary \ref{dense}).
\end{Exemp}
\vskip5mm
Using Theorem \ref{Main} and Proposition \ref{tex}, we give in the following corollary a complete characterization for the convex cone $L_c(X,Y)$ to be a vector space. The non trivial part of the following corollary is the implication $(v) \Longrightarrow (i)$, which is a consequence of Theorem ~\ref{Main}. Note also that thanks to Proposition ~\ref{tex}, we do not need to assume that $Y$ is biBanach space (for which there exists $e\in Y$ such that $\|e|_Y=1$ and $\|-e|_Y=0$) in the following corollary, since this condition  used in Theorem ~\ref{Main} is implicitly verified for the space $(\R,\|\cdot|_{\R})$ in part $(v)$. 

\begin{corollary} \label{Car} Let $(X,\|\cdot|_X)$  be an asymmetric normed space. Then,  the following assertions are equivalent.

\begin{enumerate}[$(i)$]
    \item $c(X)>0$.
    \item $(X,\|\cdot|_X)$ is isomorphic to its associated normed space.
    \item For every asymmetric normed space $(Y,\|\cdot|_Y)$, we have that $L_c(X,Y)$ is an asymmetric normed space isomorphic to the space $L(X_s,Y_s)$.
    \item $(X^{\flat},\|\cdot|_{\flat})$ is an asymmetric normed space isomorphic to the Banach space $(X^*,\|\cdot\|_{*})$.
    \item $X^{\flat}$ is a vector space.
\end{enumerate}

\end{corollary}

The following result shows that if an asymmetric normed space $X$ is a dual of some asymmetric normed space, then necessarily it is isomorphic to its associated normed space, in other words $c(X)>0$.  
\begin{corollary} Let $(X,\|\cdot|)$ be an asymmetric normed space and suppose that $c(X)=0$. Then, $X$ can not be the dual of an asymmetric normed space. The converse is false in general (ex. the Banach space $X=(c_0(\N),\|\cdot\|_{\infty})$,  is not a dual space but  $c(X)=1$). 
\end{corollary}
\begin{proof} Suppose that there exists an asymmetric normed space $Y$ such that $(Y^{\flat},\|\cdot|_{\flat})=(X,\|\cdot|)$. We prove that  $c(Y)=0$.  Indeed, suppose by contradiction that $c(Y)>0$, then by formula (\ref{eq21}) of Proposition \ref{tex} (applied with the spaces $(Y,\|\cdot|_Y)$ and $(\R,\|\cdot|_{\R})$), we have that 
\begin{eqnarray*} 
c(Y)\|p\|_{\flat}\leq \|-p|_{\flat}\leq \frac{1}{c(Y)} \|p\|_{\flat}, \hspace{2mm} \forall p\in Y^{\flat}=X.
\end{eqnarray*}
This implies that $c(X)\geq c(Y)>0$, which contradict the fact that $c(X)=0$. Hence, $c(Y)=0$. Now,  using Theorem \ref{Main}, we get that $Y^{\flat}$ is not a vector space which contradict the fact that $X=Y^{\flat}$ is a vector space. Finally, $X$ cannot be the dual of an asymmetric normed space.
\end{proof}
\begin{corollary} Let $X$ be an asymmetric normed space and $(Y,\|\cdot|_Y)$ be a biBanach asymmetric normed space for which there exists $e\in Y$ such that $\|e|_Y=1$ and $\|-e|_Y=0$. Then, either $$L_c(X,Y)\cap (-L_c(X,Y))=L_c(X,Y)\simeq L(X_s,Y_s)$$ or  $L_c(X,Y)\cap (-L_c(X,Y))$ is of first Baire category in $(L(X_s,Y_s),\|\cdot\|_{L_s})$.
\end{corollary}
\begin{proof} If $c(X)>0$, then by Corollary \ref{Car}, we have that $L_c(X,Y)$ is an asymmetric normed space isomorphic to $L(X_s,Y_s)$, thus we have that $$L_c(X,Y)\cap (-L_c(X,Y))=L_c(X,Y)\simeq L(X_s,Y_s).$$
Otherwise, we have that $c(X)=0$. In this case, to see that $L_c(X,Y)\cap (-L_c(X,Y))$ is of first Baire category in the space $(L(X_s,Y_s),\|\cdot\|_{L_s})$, it suffices to observe, using Theorem \ref{Main} with $H=0$ and $r=1$, that we have
$$L_c(X,Y)\cap (-L_c(X,Y)) =\cup \lbrace n(\overline{B}_{L_c}(0,1)\setminus \mathcal{G}(0)): n\in \N \rbrace,$$
so that, it is of first Baire category in $(L(X_s,Y_s),\|\cdot\|_{L_s})$, being the countable union of first Baire category sets.
\end{proof}
\subsection{The second main result.}
Now, we give our second main result. We are interested in the following result, concerning the density of the dual $X^{\flat}$ in $X^*$. By $ \overline{X^{\flat}}^{w^*}$, we denote the weak-star closure of $X^{\flat}$ in $(X^*,w^*)$.

\begin{theorem} \label{Tain} Let $(X,\|\cdot|)$  be an asymmetric normed space. Then,  $X$ is a $T_1$ space if and only if $\overline{X^{\flat}}^{w^*}=X^*$. 
\end{theorem}
\begin{proof} Assume that $X$ is a $T_1$ space.  Suppose by contradiction that $ \overline{X^{\flat}}^{w^*}\neq X^*$ and fix $p\in X^*\setminus \overline{X^{\flat}}^{w^*}$. By the classical Hahn-Banach theorem in the Hausdorff locally convex vector space $(X^*,w^*)$, there exists $x_0\in X\setminus \lbrace 0 \rbrace$ and $\alpha \in \R$, such that 
\begin{eqnarray} \label{Hahn-Bana}
\langle p, x_0\rangle > \alpha \geq \langle q, x_0\rangle, \textnormal{ for all } q \in \overline{X^{\flat}}^{w^*}.
\end{eqnarray}
Since $X$ is $T_1$ space and $x_0\neq 0$,  we have that $\|x_0|>0$. Now, using \cite[Theorem 2.2.2]{Cob}, there exists $q_0\in X^{\flat}$ such that $\|q_0|_{\flat}=1$ and $\langle q_0, x_0\rangle=\|x_0|$. Since $X^{\flat}\subset \overline{X^{\flat}}^{w^*}$ is a convex cone, we obtain using (\ref{Hahn-Bana}) that for all $n\in \N$,
\begin{eqnarray*} 
\langle p, x_0\rangle > \alpha \geq \langle nq_0, x_0\rangle=n\|x_0|.
\end{eqnarray*}
This implies that $\|x_0|=0$ which is impossible. Hence, $ \overline{X^{\flat}}^{w^*}=X^*$.  Conversely, suppose that $ \overline{X^{\flat}}^{w^*}=X^*$. We need to show that $\|x|>0$ whenever $x\neq 0$. Indeed, let $x\neq 0$. By the Hahn-Banach theorem (in $X^*$), there exists $p\in X^*$ such that $\|p\|_{*}=1$ and $\langle p, x\rangle=\|x\|_s>0$. On the other hand, $p\in \overline{X^{\flat}}^{w^*}=X^*$, thus, for every $\varepsilon>0$, there exists $q_\varepsilon \in X^{\flat}$ such that $$\langle q_\varepsilon, x\rangle + \varepsilon \geq \langle p, x\rangle= \|x\|_s.$$
Suppose by contradiction that $\|x|=0$. It follows that for every $\varepsilon>0$, $\langle q_\varepsilon, x\rangle\leq \|q_\varepsilon|_{\flat}\|x|=0$. So using the above formula, we get that $\|x\|_s \leq \varepsilon $ for every $\varepsilon>0$ which implies that $x=0$ and gives a contradiction. Hence, $\|x|>0$ for every $x\neq 0$, which implies that $X$ is a $T_1$ space.
\end{proof}
\begin{remark} Following the same arguments as in the above proof, we get that in general $\overline{\textnormal{span}(X^{\flat})}^{w^*}=X^*$ (even if $X$ is not a $T_1$ space).
\end{remark}

We know from \cite[Theorem 4.]{GRS} (see also \cite[Proposition 2.4.2.]{Cob})  that $\overline{B}_{\flat}(0,1)$ is always weak-star closed in $(X^*,w^*)$ (in fact, weak-star compact). On the other hand, $\overline{B}_{\flat}(0,1)$ is always  norm closed in $(X^*,\|\cdot\|_{*})$ (see Lemma \ref{lema}). These results are not always true for the whole space $ X^{\flat}$, when $c(X)=0$. We have the following characterization.
\begin{corollary} \label{clos}  Let $(X,\|\cdot|)$  be a $T_1$  asymmetric normed space. Then, $X^{\flat}$ is  weak-star closed in $(X^*,w^*)$ if and only if, $c(X)>0$ if and only if $X$ is isomorphic to its associated normed space. 
\end{corollary}
\begin{proof} If $c(X)=0$, by Theorem \ref{Main} we know that $ X^{\flat}\neq -X^{\flat}$. It follows that $X^{\flat}\neq X^*$ and so by Theorem \ref{Tain}, $ X^{\flat}\neq \overline{X^{\flat}}^{w^*}=X^*$. Equivalently, $ X^{\flat}= \overline{X^{\flat}}^{w^*}$, implies that $c(X)>0$.  Conversely, $c(X)>0$ is equivalent to the fact that $X^{\flat}$ is isomorphic to $X^*$  by Corollary \ref{Car} and so it is in particular is weak-star closed in $(X^*,w^*)$. 
\end{proof}
The space  $X^{\flat}$ (where $X=(l^{\infty}(\N^*),\|x|_{\infty})$ given is Example \ref{ex}) is not weak-star closed in 
$(X^*,w^*)$.
\begin{remark} Note that, if we assume that $X^*$ is a reflexive space, then  we can replace the $w^*$-closure by the $\|\cdot\|_{*}$-closure. This follows from the well-known Mazur's theorem on the coincidence of weak and norm topologies on convex sets (see \cite{HB}), so we have that  $ \overline{X^{\flat}}^{w^*}=\overline{X^{\flat}}^{w}=\overline{X^{\flat}}^{\|\cdot\|_{*}}$, since weak-star and weak topologies coincide in reflexive spaces.
\end{remark}

\subsection{Classification and examples} \label{Cl}
\vskip5mm
There are several topological studies of asymmetric normed spaces, see for instance \cite{A}, \cite{Cob} and \cite{JS}.  Our study leads to the classification given in Definition \ref{def1}  and the commentary which follows it (see introduction). Recall that the two possible situations which go beyond the classical framework of normed spaces are:

\begin{enumerate}[$(i)$]
    \item Infinite dimensional spaces which are $T_1$ with $c(X)=0$ (spaces of type ~II).
    \item Finite and infinite dimensional spaces $X$ which are not $T_1$ (spaces of type ~III, necessarily $c(X)=0$).
\end{enumerate}

These affirmations are consequences of  Corollary \ref{Car}, Proposition \ref{ok2} and Theorem \ref{ok3}. Let $(X,\|\cdot|)$ be  an asymmetric normed linear space endowed with the topology $\tau_{\|\cdot|}$ induced by the quasi-metric defined by 
$$d_{\|\cdot|}(x,y):=\|y-x|, \forall x, y \in X.$$
The closed unit ball $\overline{B}_{\|\cdot|_X}(0,1)$ is the set $\lbrace y\in X: \|y|\leq 1\rbrace$. 

A set $ K \subset X$ is said to be compact if it is compact considered as a subspace of $X$ with the induced topology, that is, $(K, \|\cdot|)$  is compact with respect to the topology ${\tau_{\|\cdot|}}_{|K}$. A set $K$ of $X$ is compact if every sequence in $K$ has a convergent subsequence whose limit is in $K$.

The following proposition shows that a finite dimensional asymmetric nor\-med space can never be of type ~II. 
\begin{proposition} \label{ok2} Let $(X,\|\cdot|_X)$  be an asymmetric normed space. Suppose that $X$ is of type II. Then, the closed unit balls $\overline{B}_{\|\cdot|_X}(0,1)$ and $\overline{B}_{\|\cdot\|_s}(0,1)$ of $X$ and its associated normed space respectively, are not compact, and consequently, $X$ is infinite dimensional.
\end{proposition} 
\begin{proof} From the definition of spaces of type II, there exists a sequence $(x_n)\subset X$ such that $\|x_n|_X=1$ for all $n\in \N$ and $0< \|-x_n|_X \to 0$. We can assume without loss of generality that $0< \|-x_n|_X<1$ so that $\|x_n\|_s=\|x_n|_X=1$ for all $n\in \N$. Suppose by contradiction that $B_{\|\cdot|_X}(0,1)$ is compact.  Let $(x_{n{_k}})$ be a subsequence converging for $\|\cdot|_X$ to some $a\in X$.Then, $$\|-a|_X\leq \|x_{n{_k}}-a|_X+\|-x_{n{_k}}|_X \to 0,$$ which implies that $\|-a|_X=0$. Since $X$ is a $T_1$ space, then $a=0$. This contradict the fact that $\|x_{n{_k}}-a|_X=\|x_{n{_k}}|_X=\|x_{n{_k}}\|_s=1$ for each $k\in \N$. Hence, the sequence $(x_n)$ has no convergent subsequence neither for $\|\cdot|_X$  nor for $\|\cdot\|_s$ (since $\|\cdot|_X\leq \|\cdot\|_s$). Thus, the closed unit balls $B_{\|\cdot|_X}(0,1)$ and $B_{\|\cdot\|_s}(0,1)$  are not compact. In particular $X$ is of infinite dimentional. 
\end{proof}
The following theorem shows that a $T_1$ space of finite dimension is necessarily isomorphic to its associated normed space, or equivalently, it is of type ~I.
\begin{theorem}\label{ok3}  Let $(X,\|\cdot|_X)$  be an asymmetric normed space of finite dimension. Then, $X$ is $T_1$ if and only if $X$ is of type I, if and only if $X$ is isomorphic to its associated normed space. 
\end{theorem}
\begin{proof} Suppose that $X$ is $T_1$, then, $X$ is not of type III. Since, spaces of type II are infinite dimensional by Proposition \ref{ok2}, it follows that $X$, is of type I. Hence, equivalently, by Corollary \ref{Car}, $X$ is isomorphic to its associated normed space. The converse is trivial.
\end{proof}
We recover the result of Garc\'ia-Raffi in \cite[Theorem 13.]{G} in the following corollary.
\begin{corollary} The closed unit ball of a $T_1$ asymmetric normed space $X$ is compact, if and only if it is finite dimensional. 
\end{corollary}
\begin{proof} Suppose that $X$ is finite dimensional. Since $X$ is $T_1$, then by Theorem \ref{ok3}, $X$ is isomorphic to its associated normed space. Thus, the closed unit ball of $(X,\|\cdot|_X)$  is compact. Conversely, suppose that the closed unit ball of $(X,\|\cdot|_X)$  is compact. Then, by Proposition \ref{ok2}, $X$ is not of type II. Since, $X$ is $T_1$, then $X$ is of type I and so  $(X,\|\cdot|_X)$ is isomorphic to $(X,\|\cdot\|_s)$ (by Corollary \ref{Car}), which is finite dimentional by Riesz's theorem, since its closed unit ball is compact.
\end{proof}

We give below examples corresponding to spaces of  type I, II and III. 

\begin{Exemp} \label{ex0} {\it Finite dimensional space of type III: Case where $\hat{S}_X= \emptyset$ and $c(X)=0$.}  Let $X=\R$ and $\|t|_{\R}:=\max\{0,t\}$ for all $t\in \R$. Then, $(\R, \|\cdot|_{\R})$ is an asymmetric normed space with $\hat{S}_X= \emptyset$ and $(X^{\flat},\|\cdot|_{\flat})$ is not  a vector space.
\end{Exemp}
\begin{Exemp} \label{ex1} {\it  Infinite dimensional space of type III: Case where $\emptyset \neq \hat{S}_X \subsetneq S_X$ and $c(X)=0$.}  Let $X=C_0[-1,1]$ the space of all continuous functions from $[-1,1]$ to $\R$ such that $f(0)=0$. We define on $X$ the following asymmetric norm
$$\|f|:=\sup_{x\in [-1,1]} f(x)\leq \|f\|_{\infty}=\max\{\|f|,\|-f|\}.$$
It is easy to see that $\hat{S}_X\neq \emptyset$,   $\hat{S}_X\neq S_X$ and $c(X)=0$. Let us denote by $\delta_x: X\to \R$ the evaluation map associated to $x\in [-1,1]$ defined by $\delta_x(f)=f(x)$ for all $f\in X$. Clearly, $\delta_x \in X^{\flat}$ and $\|\delta_x|_{\flat}= 1$ for all $x\in [-1,1]$. However, $-\delta_x \not \in X^{\flat}$, for all $x\in [-1,1]\setminus \lbrace 0 \rbrace$. It follows that $X^{\flat}$ is not a vector space.
\end{Exemp}
\begin{Exemp} \label{ex2} {\it  Space of type II: Case where $\hat{S}_X= S_X$ and $c(X)=0$.}  The space  $(l^{\infty}(\N^*),\|x|_{\infty})$ given is Example \ref{ex} is a space of type II.
\end{Exemp}
\begin{Exemp} \label{ex4} {\it  Space of type I: Case where $0<c(X)<1$.}  Let $(X,\|\cdot|)$ be an asymmetric normed space. Define a new asymmetric norm on $X$ as follows: $\|x|_1=\|x|+\|x\|_s$, where $\|x\|_s=\max\{\|x|,\|-x|\}$, for all $x\in X$. Then, the index of symmetry $c(X,\|\cdot|_1)$ of $X$ for the asymmetric norm $\|\cdot|_1$, satisfies $0<c(X,\|\cdot|_1)<1$. First, we see that $c(X,\|\cdot|_1)<1$ since $\|\cdot|_1$ is not a norm. Suppose that $c(X,\|\cdot|_1)=0$, there exists $(x_n)\subset X$ such that $\|x_n|+\|x_n\|_s=1$ for all $n\in \N$ and $\|-x_n|+\|-x_n\|_s\to 0$. This implies that $\|x_n\|_s=\|-x_n\|_s\to 0$. Since $\|x_n|\leq\|x_n\|_s$, it follows that $\|x_n|+\|x_n\|_s\to 0$, which is a contradiction. Recall that, in every asymmetric normed space, the condition  $c(X)>0$ implies that $\hat{S}_X=S_X$ (see Proposition \ref{index}). The dual of $(X,\|\cdot|_1)$ is a vector space by Corollary \ref{Car}.
\end{Exemp}
\vskip5mm

\section*{Acknowledgment}
The authors are grateful to the referee for these valuable remarks as well as for pointing out a mistake in a first version that we have corrected thanks to his remarks.  

This research has been conducted within the FP2M federation (CNRS FR 2036) and during a research visit of the second author to the SAMM Laboratory of the University Paris Sorbonne-Panthéon. This author thanks his hosts for hospitality. The second author was supported by the grants CONICYT-PFCHA/Doctorado Nacional/2017-21170003, FONDECYT Regular 1171854 and CMM-CONICYT AFB170001.
 
\bibliographystyle{amsplain}

\providecommand{\bysame}{\leavevmode\hbox to3em{\hrulefill}\thinspace}
\providecommand{\MR}{\relax\ifhmode\unskip\space\fi MR }
\providecommand{\MRhref}[2]{%
  \href{http://www.ams.org/mathscinet-getitem?mr=#1}{#2}
}
\providecommand{\href}[2]{#2}
\begin{thebibliography}{}

\end{thebibliography}


\begin{thebibliography}{999}
\bibitem{A} C. Alegre, {\it The weak topology in finite dimensional asymmetric normed spaces}, Topology Appl. 264 (2019) 455-461.
\bibitem{AFG1} C. Alegre, J. Ferrer, V. Gregori, {\it Quasi-uniform structures in linear lattices}, Rocky Mountain J. Math. 23 (1993) 877-884.
\bibitem{AFG2} C. Alegre, J. Ferrer, V. Gregori, {\it On the Hahn–Banach theorem in certain linear quasi-uniform structures}, Acta Math. Hungar. 82 (1999) 315-320.
\bibitem{HB} H. Brezis, Analyse fonctionnelle : théorie et applications, \'Editions Dunod, Paris, 1999.
\bibitem{DJV} A. Daniilidis, J. A. Jaramillo, F. Venegas, {\it Smooth semi-Lipschitz functions and almost isometries between Finsler Manifolds}, Arxiv: 1911.07991 (2019).
\bibitem{DSV} A. Daniilidis, J. M. Sepulcre, F. Venegas  {\it Asymmetric free spaces and canonical asymmetrizations} arXiv:2002.02647, (2020).
\bibitem{G} L. M. Garc\'ia-Raffi, {\it Compactness and finite dimension in asymmetric normed linear spaces}, Topology Appl. 153 (2005), 844-853.
\bibitem{GRS} L.M. Garc\'ia-Raffi, S. Romaguera, and E.A. S\'anchez-P\'erez, {\it The dual space of an asymmetric normed linear space}, Quaest. Math. 26 (2003), no. 1, 83-96.
\bibitem{GRS1} L. M. Garc\'ia-Raffi, S. Romaguera, and E. A. S\'anchez P\'erez, {\it On Hausdorff asymmetric
normed linear spaces}, Houston J. Math. 29 (2003), 717-728.
\bibitem{CJ} J. Cabello-S\'anchez, J. A. Jaramillo, A functional representation of almost isometries, J. Math. Anal. Appl. 445 (2017), 1243-1257.
\bibitem{Cob} S. Cobzas, Functional analysis in asymmetric normed spaces, Frontiers in Mathematics, Basel: Birkhäuser, 2013.
\bibitem{CobM} S. Cobzas, C. Mustata, {\it Extension of bounded linear functionals and best approximation in spaces with asymmetric norm}, Rev. Anal. Numer. Theor. Approx. 33 (1) (2004) 39-50.
\bibitem{JS} N. Jonard-P\'erez, E. A. S\'anchez-P\'erez, {\it Compact convex sets in 2-dimensional asymmetric normed spaces}, Quaest. Math. 39(1) (2016) 73-82.
\bibitem{PT} A. Papadopoulos, M. Troyanov, Weak Finsler structures and the Funk weak metric, Math. Proc. Camb. Phil. Soc. 147 (2) (2009), 419-437.
\bibitem{RS} S. Romaguera, M. Sanchis, {\it Semi-Lipschitz functions and best approximation in quasi-metric spaces}, J. Approx. Theory 103 (2000) 292–301.
\end{thebibliography}

\end{document}